\documentclass[12pt, reqno]{amsart}
\usepackage{amsmath, amsthm, amscd, amsfonts, amssymb, graphicx, color}
\usepackage[bookmarksnumbered, colorlinks, plainpages]{hyperref}
\usepackage{csquotes,xfrac,dsfont}
\usepackage{IEEEtrantools}
\usepackage{cases}
\usepackage{enumitem}
\usepackage{tikz}
\usetikzlibrary{matrix}

\newtheorem{theorem}{Theorem}[section]
\newtheorem{lemma}[theorem]{Lemma}

\newtheorem{corollary}[theorem]{Corollary}
\theoremstyle{definition}

\newtheorem{conjecture}[theorem]{Conjecture}

\theoremstyle{remark}
\newtheorem{remark}[theorem]{Remark}
\numberwithin{equation}{section}

\newcommand{\cF}{\mathcal{F}}

\newcommand{\bN}{\mathbb{N}}
\newcommand{\bZ}{C}

\newcommand{\D}{\mathsf{D}}

\newcommand{\s}{\mathsf{s}}
\newcommand{\auf}{[\![}
\newcommand{\zu}{]\!]}

\newcommand{\reih}[2]{\left(#1 \mid #2\right)}

\newcommand{\interv}[2]{\auf #1 \, , \, #2 \zu}

\def\le{\leqslant}
\def\ge{\geqslant}

\def\leq{\leqslant}
\def\geq{\geqslant}

\begin{document}
\setcounter{page}{1}

\centerline{}

\centerline{}


\title[The Erd\H{o}s-Ginzburg-Ziv constant of rank-two-like $p$-groups]{The Erd\H{o}s-Ginzburg-Ziv constant \\ of rank-two-like $p$-groups}

\author[B. Girard, S. Zotova]{Benjamin Girard $^1$ \MakeLowercase {and} Sofia Zotova $^2$}

\address{$^{1}$ Sorbonne Universit\'e, Universit\'e Paris Diderot, CNRS, Institut de Math\'ematiques de Jussieu - Paris Rive Gauche, IMJ-PRG, F-75005, Paris, France}
\email{\textcolor[rgb]{0.00,0.00,0.84}{benjamin.girard@imj-prg.fr}}

\address{$^{2}$ Mathematisches Institut, Universität Bonn, Bonn, Germany.}
\email{\textcolor[rgb]{0.00,0.00,0.84}{s87szoto@uni-bonn.de}}

\keywords{Additive combinatorics, Baker-Schmidt theorem, Davenport constant, Erd\H{o}s-Ginzburg--Ziv constant, finite abelian group, zero-sum sequence}

\subjclass[2020]{Primary 11B30; Secondary 05E16, 20K01}

\begin{abstract}
Adapting Reiher's proof of Kemnitz's conjecture, we obtain two refinements of a theorem of Schmid and Zhuang. Our main results provide improved upper bounds for the Erd\H{o}s-Ginzburg-Ziv constant of rank-two-like $p$-groups, and their direct products with cyclic groups of order coprime to $p$. In particular, we determine the exact value of this constant, and also confirm a conjecture of Gao, for a new infinite family of groups of arbitrarily large rank.
\end{abstract} 

\maketitle

\section{Introduction}

\medskip
Let $(G,+)$ be a finite abelian group, and let $(\mathcal{F}(G),\cdot)$ be the free abelian monoid on the set $G$. Throughout the paper, the elements of $\mathcal{F}(G)$ will simply be called \em sequences \em over $G$, and for any such element $S=g_1 \cdots g_{\ell}$, the integer $|S|=\ell$ will be called the \em length \em of $S$, and $\sigma(S)=\sum_{i=1}^{\ell} g_i \in G$ the \em sum \em of $S$.   

\medskip
A classical problem in additive combinatorics is the following. 
Given a subset $L$ of $\mathbb{N}=\{1,2,\dots\}$, what is the smallest positive integer $\mathsf{s}_L(G)$, if any exists, so that every sequence $S$ over $G$ of length $|S| \ge \mathsf{s}_L(G)$ contains a subsequence $T \mid S$ so that $\sigma(T)=0$ and $|T| \in L$?

\medskip
Specifying different values for $L$ in the above definition gives rise to a rich family of combinatorial invariants related to factorization theory \cite{GeroldingerAlfred2006NFAC,rusza}, invariant theory \cite{CDG16}, number theory \cite{carmichael}, coding theory \cite{SP11}, graph theory \cite{AFK84} and discrete geometry \cite{edel}. 

\medskip
In the present paper, we mainly focus on the interplay between three of these invariants: the \em Davenport constant \em of $G$, denoted by $\D(G)$ and defined as $\mathsf{s}_{L}(G)$ when $L=\mathbb{N}$, the constant $\eta(G)$ defined as $\mathsf{s}_{L}(G)$ when $L=\{1,\dots,\exp(G)\}$, and the \em Erd\H{o}s-Ginzburg-Ziv constant \em of $G$, denoted by $\mathsf{s}(G)$ and defined as $\mathsf{s}_{L}(G)$ when $L=\{\exp(G)\}$.

\medskip
These invariants have been studied since the early sixties \cite{EGZ,Rogers63} but remain as intriguing as ever. Various bounds and some exact values for these invariants are known that typically depend on the \em invariant factors \em of $G$,
that is to say on the unique sequence of integers $1 < n_1 \mid \cdots \mid n_r \in \mathbb{N}$ for which $G \simeq C_{n_1} \oplus \cdots \oplus C_{n_r}$, where $C_n$ denotes the cyclic group of order $n$. In this context, $r$ will be called the \em rank \em of $G$, and $n_r=\exp(G)$ the \em exponent \em of $G$.

\medskip
It readily follows from the definitions that for every finite abelian group $G$,
\begin{eqnarray} \label{trivialbounds}
2 \exp(G)-1 \le \D(G) + \exp(G)-1 \le \eta(G) + \exp(G)-1 \le \s(G).
\end{eqnarray}
It is also easy to prove that the first and second inequalities are strict unless $G$ is cyclic, and it was conjectured by Gao (see Conjecture $6.5$ in \cite{GaoGero06}) that the third one always holds as an equality.

\begin{conjecture} \label{ConjGao} For every finite abelian group $G$, one has
$$\s(G) = \eta(G) + \exp(G)-1.$$
\end{conjecture}

\medskip
In the special case of groups of rank at most two, that is to say of the form $C_m \oplus C_n$, where $1 \le m \mid n$, the invariants above are well understood, and Conjecture \ref{ConjGao} holds. 

\begin{theorem}\label{ranktwo}
Let $G \simeq C_m \oplus C_n$, where $1 \le m \mid n$ are two integers.
Then 
$$\mathsf{D}(G)=m+n-1, 
\ \ \eta(G)=2m+n-2 
\ \ \text{ and } \ \ 
\mathsf{s}(G)=2m+2n-3.$$
\end{theorem}

\medskip
The values of $\D(G)$ and $\eta(G)$ are folklore when $G$ is cyclic, that is when $m=1$. When $m > 1$, the value of $\D(G)$ was obtained by Olson \cite{OLSON2} as well as the one for $\eta(G)$ when $m=n$ is prime. An easy induction then yields the value of $\eta(G)$ for all $1 < m \mid n$ (see Theorem  $5.8.3$ in \cite{GeroldingerAlfred2006NFAC}). Concerning $\s(G)$, the result is already non-trivial in the cyclic case, and was proved by Erd\H{o}s, Ginzburg and Ziv in $1961$ \cite{EGZ}. The exact value of $\s(G)$ when $m=n$ is prime was only determined in $2003$, by Reiher \cite{Kemnitz} and di Fiore independently, thereby solving a conjecture made by Kemnitz $20$ years earlier \cite{kemnitz1983lattice}. Lifting this result to all $1 < m \mid n$ then also follows from an easy induction (see Theorem $5.8.3$ in \cite{GeroldingerAlfred2006NFAC}). For the sake of completeness, let us recall that Savchev and Chen later obtained a refinement of Reiher's theorem \cite{savchev2005kemnitz}. 

\medskip
In the case of groups of rank at least three, far less is known and the picture, already in the special case of finite abelian $p$-groups, shows a lot more contrast. 

\medskip 
On the one hand, the exact value of $\D(G)$ was determined in $1969$ for all $p$-groups by Olson \cite{OLSON1} and Kruyswijk \cite{EmdeBoas69} independently.

\begin{theorem}
\label{Olson1}
    Let $G \simeq \bZ_{p^{a_1}}\oplus\dots\oplus\bZ_{p^{a_r}}$, where $p$ is prime and $a_1,\dots,a_r$ are positive integers, be a finite abelian $p$-group. Then,
    \begin{align}
        \D(G)=\sum_{i=1}^r(p^{a_i}-1) + 1\,.
    \end{align}
\end{theorem}

On the other hand, and as far as groups of rank at least three are concerned, the exact values of $\eta(G)$ and $\s(G)$ are known only for very special types of $p$-groups, such as a) those of the form $C^r_p$ where $p=3$ and $r \in \{3,4,5,6\}$ or $p=5$ and $r=3$ \cite{edel}, and b) homocyclic $2$-groups (see Satz $1$ in \cite{harborth1973extremalproblem}) and closely related ones (see Corollary $4.4$ in \cite{edel}).

\medskip
In general, it is rather easy to see (Lemma $3.2$ in \cite{edel}) that when $G$ is a finite abelian $p$-group, the second inequality in (\ref{trivialbounds}) can be improved to 
\begin{eqnarray}\label{boundetap}
2\D(G)-\exp(G) \le \eta(G),
\end{eqnarray}
which, by the third inequality in (\ref{trivialbounds}), leads to 
\begin{eqnarray}\label{boundsp}
2\D(G)-1 \le \s(G).
\end{eqnarray}

In $2010$, Schmid and Zhuang \cite{zhuang} conjectured that the inequality (\ref{boundsp}), and hence the inequality (\ref{boundetap}), are in fact equalities for all finite abelian $p$-groups such that $\D(G) \le 2\exp(G)-1$. Such groups will be called \em rank-two-like \em $p$-groups, since it follows easily from Theorem \ref{ranktwo} that any group $G$ of rank at most two satisfies $\D(G) \le 2\exp(G)-1$ indeed. Also, note that rank-two-like $p$-groups can have an arbitrarily large rank, and thus form an interesting class of groups to investigate.

\medskip
In support of their conjecture, Schmid and Zhuang obtained the following result (see Theorem $1.2$ in \cite{zhuang}).

\begin{theorem}\label{SchmidZhuang}
 Let $p \ge 3$ be a prime. If $G$ is a finite abelian $p$-group such that $\D(G)\le 2\exp(G)-1$, then
 $$\s(G)\le\D(G)+2\exp(G)-2.$$
\end{theorem}

Note that if $\D(G)=2\exp(G)-1$, Theorem \ref{SchmidZhuang} gives equality in (\ref{boundsp}), so that the exact values of $\eta(G)$ and $\s(G)$ follow and satisfy Conjecture \ref{ConjGao} indeed. Therefore, and since $\D(C_p \oplus C_p)=2p-1$, Theorem \ref{SchmidZhuang} can be seen as an extension of Reiher's theorem on the Erd\H{o}s-Ginzburg-Ziv constant of rank-two groups of the form $C_p \oplus C_p$ when $p \ge 3$ is prime.

\medskip
Subsequently, Luo could prove that, as conjectured by Schmid and Zhuang, there is equality in (\ref{boundetap}) for all rank-two-like $p$-groups (see Theorem $1.6$ in \cite{luo}).

\begin{theorem}\label{thm:Luo}
    Let $p$ be a prime. If $G$ is a finite abelian $p$-group such that $\D(G)\le 2\exp(G)-1$, then 
    $$\eta(G)=2\D(G)-\exp(G).$$
\end{theorem}

Note that while Theorem \ref{SchmidZhuang} applies to odd primes only, Theorem \ref{thm:Luo} holds when $p=2$ also. 
  
\section{New results and plan of the paper}

In this paper, we refine Theorem \ref{SchmidZhuang} in two ways. We do so by adapting the original argument used by Reiher \cite{Kemnitz} to prove Kemnitz's conjecture, one of the main changes here being the use of Baker-Schmidt theorem (see Theorem $2$ in \cite{BS}) in place of Chevalley-Warning theorem. 

\medskip
Our first main result is the following.

\begin{theorem}\label{thm:thmgal}
    Let $p \ge 3$ be a prime. If $G$ is a finite abelian $p$-group such that $\D(G)\le 2\exp(G)-p^k$ for some prime power $1 \le p^{k} \le \exp(G)$, then 
    $$\s(G)\leq\D(G)+2\exp(G)-p^k-1.$$
\end{theorem}

This theorem has a certain number of interesting corollaries that we now proceed to state and quickly discuss. 

\medskip
Firstly, the following corollary gives the exact value of the Erd\H{o}s-Ginzburg-Ziv constant, and thereby confirms Conjecture \ref{ConjGao}, for a new infinite family of groups having arbitrarily large rank.

\begin{corollary}\label{cor:sbestimmen}
    Let $p \ge 3$ be a prime. If $G$ is a finite abelian $p$-group such that $\D(G)=2\exp(G)-p^k$ for some prime power $1 \le p^k \le \exp(G)$, then
    \begin{align*}
        \s(G)=2\D(G)-1\,.
    \end{align*} 
\end{corollary}

Note that groups satisfying the assumptions of Corollary \ref{cor:sbestimmen} abound. Indeed, for any prime $p$ and any positive integer $k$, it follows from Theorem \ref{Olson1} that the $p$-group $G \simeq \bZ_p^{p^k}\oplus\bZ_{p^{k+1}}$ verifies $\D(G)=2\exp(G)-p^k$, where $1 \le p^k \le \exp(G)$.

\medskip
Secondly, and in the case where the exact value of the Erd\H{o}s-Ginzburg-Ziv constant remains unknown, we are able to improve on the upper bound provided by Theorem \ref{SchmidZhuang}.

\begin{corollary}\label{cor:<}
    Let $p \ge 3$ be a prime. If $G$ is a finite abelian $p$-group such that $\D(G)<2\exp(G)-1$, then 
    $$\s(G)\leq\D(G)+2\exp(G)-p-1.$$
\end{corollary}

Finally, a routine argument allows one to extend the reach of our Theorem \ref{thm:thmgal} to all direct products of any rank-two-like $p$-group with a cyclic group of order coprime to $p$. This gives the following generalization of Theorem \ref{thm:thmgal}.

\begin{theorem}\label{thm:pgpe+cycl} Let $p \ge 3$ be a prime. If $H$ is a finite abelian $p$-group such that $\D(H) \le 2\exp(H)-p^k$ for some prime power $1 \le p^k \le \exp(H)$, and if $a$ is a positive integer coprime to $p$, then $G \simeq H\oplus\bZ_{a}$ satisfies
\begin{align}
    \s(G)\leq \D(H)+2\exp(G)-p^k-1\,.\label{eq:sleqH+C_n}
\end{align}
\end{theorem}

As a corollary, we derive the exact value of the Erd\H{o}s-Ginzburg-Ziv constant for any direct product of a finite abelian $p$-group $H$ such that $\D(H)=2\exp(H)-p^k$ for some prime power $1 \le p^k \le \exp(H)$ with a cyclic group of order coprime to $p$. This settles Conjecture \ref{ConjGao} for all finite abelian groups of this type.

\begin{corollary}\label{cor:pgpe+cycl}
Let $p \ge 3$ be a prime. If $H$ is a finite abelian $p$-group such that $\D(H)=2\exp(H)-p^k$ for some prime power $1 \le p^k \le \exp(H)$, and if $a$ is a positive integer coprime to $p$, then $G \simeq H\oplus\bZ_{a}$ satisfies $$\s(G)=\D(H)+2\exp(G)-p^k-1=2\D(G)-1.$$
\end{corollary}

Using the same overall approach, but with a slight twist, we obtain the following as our second main result. 
\begin{theorem}\label{thm:thmc}
    Let $p \ge 3$ be a prime. If $G$ is a finite abelian $p$-group such that $\D(G)=2\exp(G)-c$ for some $1\le c\le \exp(G)$. Then, one has
    $$\s(G)\leq\D(G)+2\exp(G)-\left(\frac{c-1}{2}\right)-2.$$
\end{theorem}

The outline of the paper is as follows. In Section \ref{BS&co}, the Baker-Schmidt theorem is applied to obtain useful identities modulo $p$ involving the numbers of zero-sum subsequences of each length in any long enough sequence over a rank-two-like $p$-group. In passing, one of these identities will provide a new short proof of a key element in Luo's proof of Theorem \ref{thm:Luo}. In Section \ref{combtricks}, we then proceed to the proofs of Theorem \ref{thm:thmgal} and its corollaries. In Section \ref{extraresult}, we prove Theorem \ref{thm:thmc}, and in Section \ref{conclu}, a few concluding remarks will be made.

\section{Baker-Schmidt Theorem and useful corollaries}
\label{BS&co}

For any two integers $a\leq b$, let us set $\interv{a}{b}=\{x \in \mathbb{Z} : a \le x \le b\}$. Now, let $p$ be a prime and $\ell,s$ be two positive integers. Given a family $G_1,\dots,G_{\ell}$ of finite abelian $p$-groups, and a family $\cF_1,\dots,\cF_{\ell}$ of polynomials such that $\cF_i \in G_i[x_1,\dots,x_s]$ for each $i \in \interv{1}{\ell}$, we are interested in the solutions $\varepsilon \in \{0,1\}^s$ to the system
$$\cF_i(\varepsilon)=0 \text{ in } G_i,\ \text{ for every } i \in \interv{1}{\ell}.$$ 

Each such solution $\varepsilon=
(\varepsilon_1,\dots,\varepsilon_s)$ has a \em support \em  $S(\varepsilon)=\{i \in \interv{1}{s} : \varepsilon_i=1\}$ and a \em weight \em $\text{w}(\epsilon)=|S(\varepsilon)|=\sum^s_{i=1} \varepsilon_i$. The total number of solutions $\varepsilon \in \{0,1\}^s$ of even (resp. odd) weight to the above system will be denoted by $A(\cF_1,\dots,\cF_\ell)$ (resp. $B(\cF_1,\dots,\cF_\ell))$.

\medskip
We are now ready to state the Baker-Schmidt theorem (Theorem $2$ in \cite{BS}), a useful extension of Theorem \ref{Olson1} which will be key to our purpose.

\begin{theorem}\label{thm:BS}Let $p$ be a prime and $\ell$ be a positive integer. For every $i \in \interv{1}{\ell}$, let $G_i$ be a finite abelian $p$-group, and let $\cF_i$ be a polynomial in $G_i[x_1,\dots,x_s]$ of total degree $d_i$. If
    \begin{align}
        s \ge \sum_{i=1}^{\ell} d_i(\D(G_i)-1) + 1\,,\label{eq:diD}
    \end{align}
    then
    \begin{align*}
        A(\cF_1,\dots,\cF_\ell)-B(\cF_1,\dots,\cF_\ell) \equiv0\mod p\,.
    \end{align*}
\end{theorem}

Before deducing from Theorem \ref{thm:BS} an important lemma that will be at the core of our proofs, we recall a notation that was originally introduced in \cite{Kemnitz}.

\medskip
Given a finite abelian group $G$, a sequence $X\in \cF(G)$ and an integer $k$, we denote by $\reih{k}{X}$ the number of subsequences $Y\mid X$ of length $|Y| = k$ such that $\sigma(Y)=0$.
In particular, note that $\reih{0}{X}=1$ and that $\reih{k}{X}=0$ whenever $k < 0$ or $k > |X|$.

    \begin{lemma}\label{lm:lmdebase}
        Let $p$ be a prime and $G$ be a finite abelian $p$-group with $\exp(G)=n$. Moreover, let $\gamma,\beta\geq0$ and $k\geq 2$ be integers. If $J\in \cF(G)$ is a sequence of length $t\in\auf\D(G)+n-1-\gamma, \, kn-1-\gamma-\beta\zu$, then one has
        \begin{align}
            \sum_{j=0}^{k-1}(-1)^j\left(\sum_{i=0}^{\gamma}\binom{\gamma}{i}\reih{jn-i-\beta}{J}\right)\equiv0\mod p\,.\label{eq:BSlmgeneral}
        \end{align}
    \end{lemma}

\begin{proof}
Write $J=g_1\cdots g_t$ and consider the system in $t+\gamma$ variables consisting of the following two polynomial equations of degree one:

\begin{IEEEeqnarray}{rC,C"t,L}
     &\sum_{i=1}^{t}x_i+\sum_{j=1}^{\gamma}x_{t+j} + \beta & = \, 0 & in & \bZ_n\,,\label{eq:betagammaZ}
     \\*[0.4\normalbaselineskip]
    \smash{\left\{
      \IEEEstrut[14\jot]
    \right.} \nonumber\\
   &\sum_{i=1}^{t}g_ix_i  & = \, 0 & in & G\,.\label{eq:betagammaG} 
\end{IEEEeqnarray}
First, a tuple $\varepsilon=(\varepsilon_1,\dots,\varepsilon_{t+\gamma}) \in \{0,1\}^{t+\gamma}$ is a solution to \eqref{eq:betagammaZ} if and only if $\text{w}(\varepsilon)\equiv-\beta\mod n$, that is $\text{w}(\varepsilon) = jn-\beta$ for some $0\leq j\leq k-1$ (since by hypothesis, $\text{w}(\varepsilon) \le t + \gamma \le (kn-1-\gamma-\beta)+\gamma<kn-\beta$). 
Therefore, the solutions $\varepsilon=(\varepsilon_1,\dots,\varepsilon_{t+\gamma})\in \{0,1\}^{t+\gamma}$ to the above system are exactly those the weight of which is $\text{w}(\varepsilon) = jn-\beta$ for some $0\leq j\leq k-1$ and the first $t$ coordinates of which satisfy $\sum^t_{i=1}g_i\varepsilon_i=0$. Also, note that the weight of $(\varepsilon_1,\dots,\varepsilon_t)$ varies between $\text{w}(\varepsilon)-\gamma$ and $\text{w}(\varepsilon)$.

Now, for each $j\in \interv{0}{k-1}$ and each $i\in\interv{0}{\gamma}$, there are $\reih{jn-\beta-i}{J}$ tuples $(\varepsilon_1,\dots,\varepsilon_t)$ of weight $jn-\beta-i$ satisfying (\ref{eq:betagammaG}), and for each such tuple, there are $\binom{\gamma}{i}$ ways to choose $(\varepsilon_{t+1},\dots,\varepsilon_{t+\gamma}) \in \{0,1\}^{\gamma}$ so that $(\varepsilon_1,\dots,\varepsilon_{t+\gamma})$ is a solution to the above system (the only constraint on such a $(\varepsilon_{t+1},\dots,\varepsilon_{t+\gamma})$ being that its weight must be equal to $i$).
Thus, for each $j\in \interv{0}{k-1}$, the total number of solutions $\varepsilon \in \{0,1\}^{t+\gamma}$ of weight $\text{w}(\epsilon)=jn-
\beta$ to the above system is $\sum_{i=0}^{\gamma}\binom{\gamma}
{i}\reih{jn-\beta-i}{J}$.

Since
\begin{align*}
   t+\gamma
    &\ge \D(G)+n-1 \\
    &>\D(G)+n-2\\
    &= 1\cdot(\D(G)-1)+1\cdot(\D(\bZ_n)-1)\,,
\end{align*}
it follows from Theorem \ref{thm:BS} that
$$\sum_{j=0}^{k-1}(-1)^{jn-\beta}\left(\sum_{i=0}^{\gamma}\binom{\gamma}{i}\reih{jn-i-\beta}{J}\right)\equiv0\mod p\,.$$
If $n$ is even, then $p=2$ and we don't mind the signs. If $n$ is odd, then $jn$ has the same parity as $j$. In both cases, the desired result is proved.
\end{proof}

Note that whenever $p$ is prime and $n$ is a power of $p$, applying Lemma \ref{lm:lmdebase} with $k=2$ and $\gamma=\beta=0$ to any sequence $J$ over $\bZ_n$ of length $|J|=2n-1$ yields $\reih{n}{J} \equiv 1 \mod p$ and thus proves the Erd\H{o}s-Ginzburg-Ziv theorem.
Most importantly, Lemma \ref{lm:lmdebase} implies the following.

\begin{corollary}\label{BSlemma}
    Let $p$ be a prime and $G$ be a finite abelian $p$-group with $\exp(G)=n$. The following four statements hold.
    \begin{enumerate}[label=\alph*)]
        \item Let $\gamma\ge0$ be an integer, and $J\in \cF(G)$ be a sequence of length $t\in\auf\D(G)+n-1-\gamma, \, 3n-1-\gamma\zu$. Then, one has
        \begin{align}
             1-\left(\sum_{i=0}^{\gamma}\binom{\gamma}{i}\reih{n-i}{J}\right)+\left(\sum_{i=0}^{\gamma}\binom{\gamma}{i}\reih{2n-i}{J}\right)\equiv0\mod p\,.\label{eq:a}
        \end{align}
                \item Let $\gamma\ge0$ and $\beta \ge 1$ be integers, and $J\in \cF(G)$ be a sequence of length $t\in\auf\D(G)+n-1-\gamma, \, 3n-1-\gamma-\beta\zu$. Then, one has
        \begin{align}
             \left(\sum_{i=0}^{\gamma}\binom{\gamma}{i}\reih{n-i-\beta}{J}\right)-\left(\sum_{i=0}^{\gamma}\binom{\gamma}{i}\reih{2n-i-\beta}{J}\right)\equiv0\mod p\,.\label{eq:betagamma}
        \end{align}
        \item Let $J\in \cF(G)$ be a sequence of length $t\in\auf\D(G)+n-1,4n-1\zu$. Then, one has
        \begin{align}
            1-\reih{n}{J}+\reih{2n}{J}-\reih{3n}{J}\equiv0\mod p\,.\label{eq:d}
        \end{align}
        \item Let $\beta\ge 1$ be an integer, and $J\in \cF(G)$ be a sequence of length $t\in\auf\D(G)+n-1,4n-1-\beta\zu$. Then, one has
        \begin{align}
            \reih{n-\beta}{J}-\reih{2n-\beta}{J}+\reih{3n-\beta}{J}\equiv0\mod p\,.\label{eq:e}
        \end{align}
    \end{enumerate}
\end{corollary}
\begin{proof}
    \emph{a)} Apply Lemma \ref{lm:lmdebase} with $k=3$ and $\beta=0$.

    \medskip
    
    \emph{b)} Apply Lemma \ref{lm:lmdebase} with $k=3$.
    
    \medskip

    \emph{c)} Apply Lemma \ref{lm:lmdebase} with $k=4$ and $\gamma=\beta=0$.

    \medskip

    \emph{d)} Apply Lemma \ref{lm:lmdebase} with $k=4$ and $\gamma=0$.
\end{proof}

\begin{corollary}\label{cor}
    Let $p$ be a prime, $G$ be a finite abelian $p$-group with $\exp(G)=n$ and $J\in\cF(G)$ be a sequence of length $t\in\auf\D(G)+n-1,3n-1\zu$. If $\reih{n}{J}\equiv0\mod p$, then $\reih{2n}{J}\equiv-1\mod p$.
\end{corollary}
\begin{proof}
    Set $\gamma=0$ in Corollary \ref{BSlemma} a).
\end{proof}

\medskip
To conclude this section, we notice that Corollary \ref{BSlemma} allows one to easily prove the following statement (Lemma $3.2$ in \cite{luo}) which plays an important role in the proof of Luo's Theorem \ref{thm:Luo}.

\begin{corollary}
    Let $p$ be a prime and $G$ be a finite abelian $p$-group with $\exp(G)=n$. Then, for every $j\in\interv{1}{n}$, one has $$\s_{\interv{j}{n}+\bN}(G)=\D(G)+j-1.$$
\end{corollary}

\begin{proof} Let $j\in\interv{1}{n}$. On the one hand, if $S$ is a sequence of length $|S| = \D(G)-1$ containing no non-empty zero-sum subsequence, then $T=0^{j-1}S$ has length $|T| = \D(G)+j-2$ and the only non-empty zero-sum subsequences of $T$ are those of the form $0^i$, where $i\in\interv{1}{j-1}$. It follows that  $\s_{\interv{j}{n}+\bN}(G)\geq\D(G)+j-1$.
 
    Now, assume for a contradiction that there exists a sequence $T$ of length $|T| = \D(G)+j-1$ containing no zero-sum subsequence of length $t\in\interv{j}{n}+\bN$. It follows that $\reih{qn-i}{T}=0$ for every $q \in \mathbb{N}$ and every $i\in\interv{0}{n-j}$. Applying Corollary \ref{lm:lmdebase} with $\gamma=n-j$, $\beta=0$ and any large enough integer $k$ then gives $1 \equiv 0 \mod p$, a contradiction.
\end{proof}

\section{Proofs of Theorem \ref{thm:thmgal} and its corollaries}
\label{combtricks}

Let us start with a classical theorem due to Lucas \cite{lucas1878congruences}.

\begin{theorem}
\label{thm:Lucas}
    Let $p$ be a prime, and $r$ be an integer that is large enough so that we can write $m=m_0+\dots+m_rp^r$ and $k=k_0+\dots+k_rp^r$ with all $m_i$ and $k_i$ in $\interv{0}{p-1}$. Then, one has
    \begin{align*}
        \binom{m}{k}\equiv\prod_{i=0}^r\binom{m_i}{k_i}\mod p\,.
    \end{align*}
\end{theorem}

\begin{remark}\label{rq:binommodp}
    Note that by Theorem \ref{thm:Lucas}, $\binom{hn+a}{n}\equiv h\mod p$ holds for every $a\in\interv{0}{n-1}$ when $h$ is a positive integer.
\end{remark}

Theorem \ref{thm:Lucas} will be used in the proof of the following lemma, which will allow us to derive analogues of identity \eqref{eq:BSlmgeneral} for long sequences from knowledge about shorter ones. 

\begin{lemma}\label{binom}
    Let $p$ be a prime and $G$ be a finite abelian $p$-group with $\exp(G)=n$. Let $\gamma\in\interv{0}{n-1}$ be an integer and suppose that we have $n+1+\gamma\leq\D(G)\leq2n$. If $X\in\cF(G)$ is a sequence of length $|X|\in\auf\D(G)+2n-1-\gamma,4n-1-\gamma\zu$, then
    \begin{align}
        3-2\left(\sum_{i=0}^{\gamma}\binom{\gamma}{i}\reih{n-i}{X}\right)+\left(\sum_{i=0}^{\gamma}\binom{\gamma}{i}\reih{2n-i}{X}\right)\equiv0\mod p\,.\label{eq:binomlemma}
    \end{align}
\end{lemma}
\begin{proof} Let $X\in\cF(G)$ be a sequence of length $|X|\in\auf\D(G)+2n-1-\gamma,4n-1-\gamma\zu$, and $I\mid X$ be a subsequence of length $|I| = |X|-n$. 
    Since by hypothesis, we have $|X|-n\in\auf\D(G)+n-1-\gamma,3n-1-\gamma\zu$, we can apply Corollary \ref{BSlemma} a), so that $I$ satisfies 
    $1-\left(\sum_{i=0}^{\gamma}\binom{\gamma}{i}\reih{n-i}{I}\right)+\left(\sum_{i=0}^{\gamma}\binom{\gamma}{i}\reih{2n-i}{I}\right)\equiv0\mod p$. 
      We obtain
    \begin{align*}
        0\equiv& \ \sum_{\substack{I\mid X\\|I|=|X|-n}}\left(1-\left(\sum_{i=0}^{\gamma}\binom{\gamma}{i}\reih{n-i}{I}\right)+\left(\sum_{i=0}^{\gamma}\binom{\gamma}{i}\reih{2n-i}{I}\right)\right)\\
        =& \ \binom{|X|}{n}-\left(\sum_{i=0}^{\gamma}\binom{\gamma}{i}\binom{|X|-n+i}{n}\reih{n-i}{X}\right)\\
        &+\left(\sum_{i=0}^{\gamma}\binom{\gamma}{i}\binom{|X|-2n+i}{n}\reih{2n-i}{X}\right)\\
        =& \ 3-2\left(\sum_{i=0}^{\gamma}\binom{\gamma}{i}\reih{n-i}{X}\right)+\left(\sum_{i=0}^{\gamma}\binom{\gamma}{i}\reih{2n-i}{X}\right)\mod p\,.
    \end{align*}
    Here, the last congruence follows from Remark \ref{rq:binommodp}, that implies $\binom{s+i}{n}\equiv h\mod p$ for every $s\in\auf hn,(h+1)n-1-\gamma\zu$ when $h$ is a positive integer, for all $i\in\auf0,\gamma\zu$. Indeed, we have $|X|\in\auf\D(G)+2n-1-\gamma,4n-1-\gamma\zu\subseteq\auf3n,4n-1-\gamma\zu$, so that $|X|-n\in\auf2n,3n-1-\gamma\zu$ and $|X|-2n\in\auf n,2n-1-\gamma\zu$. 
\end{proof}

The following lemma is similar in spirit to Lemma \ref{binom}, but deals with the case where the Davenport constant of $G$ is at most $n+\gamma$. For the sake of simplicity, and since this lemma will later be used in a very special case only, we choose not to write it in full generality. 

\begin{lemma}\label{binom_pte_dav}
    Let $p$ be a prime, and $G$ be a finite abelian $p$-group with $\exp(G)=n$. Moreover, suppose that $n+1 \le \D(G)\leq n+p^k$ for some prime power $1 \le p^k \le n$. If $X\in\cF(G)$ is a sequence of length $|X|\in\auf\D(G)+2n-1-p^k,\,3n-1\zu$ such that $\reih{n}{X}=0$, then
    \begin{align}
        2-2\reih{n-p^k}{X}+\reih{2n-p^k}{X}\equiv0\mod p\,.\label{eq:321}
    \end{align}
\end{lemma}
\begin{proof}
    Let $X\in\cF(G)$ be a sequence satisfying the assumptions of the lemma, and $I\mid X$ be a subsequence of length $|I| = |X|-n$. 
    Since $|X|-n\in\auf\D(G)+n-1-p^k, \, 2n-1\zu$ and $2n-1 \le 3n-1-p^k$, Corollary \ref{BSlemma} a) applies with $\gamma=p^k$. Using the fact that, by Theorem \ref{thm:Lucas}, $\binom{p^k}{i}\equiv0\mod p$ for every $i\in\interv{1}{p^k-1}$, we obtain
    \begin{align*}
        0&\equiv1-\left(\sum_{i=0}^{p^k}\binom{p^k}{i}\reih{n-i}{I}\right)+\left(\sum_{i=0}^{p^k}\binom{p^k}{i}\reih{2n-i}{I}\right)\\
        &\equiv1-\reih{n}{I}-\reih{n-p^k}{I}+\reih{2n}{I}+\reih{2n-p^k}{I}\\
    & \equiv1-\reih{n-p^k}{I}+\reih{2n-p^k}{I}\mod p\,,
    \end{align*}
where, for the last congruence, we used $\reih{n}{I}=0$ (this follows from $\reih{n}{X}=0$) as well as $\reih{2n}{I}=0$ (this follows from $|I| = |X| -n \leq 2n-1$). 

Now, summing over all subsequences $I \mid X$ of length $|I| = |X|-n$ yields
    \begin{align*}
        0&\equiv\sum_{\substack{I\mid X\\|I|=|X|-n}}\left(1-\reih{n-p^k}{I}+\reih{2n-p^k}{I}\right)\\
        &=\binom{|X|}{n}-\binom{|X|-n+p^k}{n}\reih{n-p^k}{X}
        +\binom{|X|-2n+p^k}{n}\reih{2n-p^k}{X}\\
        &\equiv2-2\reih{n-p^k}{X}+\reih{2n-p^k}{X}\mod p\,,
    \end{align*}
    where the last congruence follows from Remark \ref{rq:binommodp}, as we have $|X|\in\auf\D(G)+2n-1-p^k, \, 3n-1\zu\subseteq\auf3n-1-p^k, \, 3n-1\zu\subseteq\interv{2n}{3n-1}$ (since $n+1 \le \D(G)$) and thus $|X|-n+p^k\in\auf\D(G)+n-1, \, 2n-1+p^k\zu\subseteq\auf 2n, \, 2n-1+p^k\zu\subseteq\auf2n, \, 3n-1\zu$ (since $n+1 \le \D(G)$) as well as $|X|-2n+p^k\in\auf\D(G)-1, \, n-1+p^k\zu\subseteq\auf n, \, 2n-1\zu$.
\end{proof}

Another key lemma for the proofs of our Theorems \ref{thm:thmgal} and \ref{thm:thmc} is the following.

\begin{lemma}\label{ABC}
    Let $p$ be a prime and $G$ be a finite abelian $p$-group with $\exp(G)=n$. Let $J\in\cF(G)$ be a sequence with $\reih{n}{J}=0$. 
    Let also $\alpha\in\interv{1}{n}$ be an integer such that $\alpha\leq2n+1-\D(G)$ and $|J|\in\auf\D(G)+2n-\alpha-1,4n-\alpha-1\zu$. 
    Then we have $\reih{n-\alpha}{J}\equiv\reih{3n-\alpha}{J}\mod p$.
\end{lemma}

\begin{proof}
     Let $N$ be the number of different ways to write $J=A B C$ so that $|A|=n-\alpha$, $|C|=2n$ and $\sigma(A)=\sigma(C)=0$. Let also $JA^{-1}$ (resp. $JB^{-1}$) denote the only sequence $T_A$ (resp. $T_B$) in $\cF(G)$ satisfying $AT_A=J$ (resp. $BT_B=J$). Note that $\reih{n}{J}=0$ implies $\reih{n}{JA^{-1}}=0$ and $\reih{n}{JB^{-1}}=0$. 
     
     On the one hand, since $|J A^{-1}|=|J|-n+\alpha\in\auf\D(G)+n-1,3n-1\zu$ and $\reih{n}{J A^{-1}}=0$, it follows from Corollary \ref{cor} that
    \begin{align*}
        N&=\sum_{\substack{A\mid J \\ |A|=n-\alpha \\ \sigma(A)=0}}\sum_{\substack{C\mid J A^{-1} \\ |C|=2n \\ \sigma(C)=0}}1\\
        &=\sum_{\substack{A\mid J \\ |A|=n-\alpha \\ \sigma(A)=0}}\reih{2n}{J A^{-1}}\\
        &\equiv\sum_{\substack{A\mid J \\ |A|=n-\alpha \\ \sigma(A)=0}}(-1)\\
        &=-\reih{n-\alpha}{J}\mod p.
    \end{align*}

    \medskip

    On the other hand, since $|JB^{-1}|=3n-\alpha\in\auf\D(G)+n-1,3n-1\zu$ and $\reih{n}{J B^{-1}}=0$, it follows from Corollary \ref{cor} that
    \begin{align*}
        N&=\sum_{\substack{B\mid J \\ |B|=|J|-3n+\alpha \\ \sigma(JB^{-1})=0}}\sum_{\substack{C\mid J B^{-1} \\ |C|=2n \\ \sigma(C)=0}}1\\
        &=\sum_{\substack{B\mid J \\ |B|=|J|-3n+\alpha \\ \sigma(JB^{-1})=0}}\reih{2n}{J B^{-1}}\\
        &\equiv\sum_{\substack{B\mid J \\ |B|=|J|-3n+\alpha \\ \sigma(JB^{-1})=0}}(-1)\\
        &=\sum_{\substack{B'\mid J \\ |B'|=3n-\alpha \\ \sigma(B')=0}}(-1)\\
        &=-\reih{3n-\alpha}{J}\mod p\,.
    \end{align*}
    
    These two identities put together give $\reih{n-\alpha}{J}\equiv-N\equiv\reih{3n-\alpha}{J}\mod p$, which is the desired result.
\end{proof}

Finally, we will need the following result generalizing Lemma $3.2$ in \cite{5pves}.

\begin{lemma}\label{Dubiner}
    Let $p$ be a prime and $G$ be a finite abelian $p$-group with $\exp(G)=n$. If $\D(G)\leq2n$ and $J\in\cF(G)$ is a zero-sum sequence of length $|J|=3n$, then $\reih{n}{J}\neq0$.
\end{lemma}
\begin{proof}
    Let $J$ be a zero-sum sequence of length $3n$ and $x$ be any element of $J$. Assume for a contradiction that $\reih{n}{J}=0$ and denote by $Jx^{-1}$ the only sequence $T_x \in \cF(G)$ such that $xT_x=J$. Since $|Jx^{-1}|=3n-1$, Corollary \ref{cor} implies $\reih{2n}{Jx^{-1}}\equiv-1\mod p$. It follows that $Jx^{-1}$, and hence $J$ itself, contains a zero-sum subsequence $U$ of length $2n$. The only sequence $T$ in $\cF(G)$ such that $TU=J$ then satisfies $\sigma(T)=0$ and $|T|=3n-2n=n$, so that $\reih{n}{J}\neq0$, a contradiction.
\end{proof}

\medskip
Using these lemmas, we are now able to prove Theorem \ref{thm:thmgal}.

\medskip
\noindent
\emph{Proof of Theorem \ref{thm:thmgal}.} Set $n=\exp(G)$. We show that every sequence in $\cF(G)$ of length $\D(G)+2n-p^k-1$ contains a zero-sum subsequence of length $n$. Let $X$ be such a sequence and assume for a contradiction that $\reih{n}{X}=0$. Note that in particular, this implies $\reih{n}{J}=0$ for every $J\mid X$.

    Firstly, we may suppose that $n+1 \le \D(G)$. Indeed, according to the remark made right after inequality (\ref{trivialbounds}), the inequality $\D(G)\leq n$ would imply that $G$ is cyclic, in which case the result directly follows from Theorem \ref{ranktwo}. 
    
    Secondly, if we had $\reih{3n}{X}\neq0$, then we would be able to find $J\mid X$ with $|J|=3n$ and $\sigma(J)=0$ and so by Lemma \ref{Dubiner} (which can be applied as $\D(G)\leq2n$), we would have $\reih{n}{J}\neq0$, a contradiction. So, from now on, we may also suppose that $\reih{3n}{X}=0$.

    Since $|X|=\D(G)+2n-p^k-1\leq 4n-2p^k-1\leq4n-1-p^k$, we have $|X|\in\auf\D(G)+2n-1-p^k,\ 4n-1-p^k\zu$. Moreover, we have $2n+1-\D(G) \ge p^k+1$. Since we also have $\reih{n}{X}=0$ by assumption, we may apply Lemma \ref{ABC} with $\alpha=p^k$, which yields
    \begin{align}
        \reih{n-p^k}{X}\equiv\reih{3n-p^k}{X}\mod p\,.\label{eq:n-cvs3n-c}
    \end{align}

    We now consider two cases, depending on whether $n+p^k+1\leq\D(G)$ or not.

    \medskip
    
    Suppose first $n+p^k+1\leq\D(G)$. We have $\D(G)\leq2n$ and as already mentioned, $|X|\in\auf\D(G)+2n-1-p^k,4n-1-p^k\zu$. Therefore, we can apply Lemma \ref{binom} with $\gamma=p^k$. In addition, it follows from Theorem \ref{thm:Lucas} that $\binom{p^k}{i}\equiv0\mod p$ for all $i\in\interv{1}{p^k-1}$ . Taking this and the fact that $\reih{n}{X}=0$ into account, \eqref{eq:binomlemma} becomes
    \begin{align}
        0\equiv& \ 3-2\reih{n-p^k}{X}+
        \reih{2n}{X}+\reih{2n-p^k}{X}\mod p\,.\label{eq:321clean}
    \end{align}    
    Finally, since $\D(G)+2n-p^k-1 \ge \D(G)+n-1$, we have $|X|\in\auf\D(G)+n-1,\ 4n-1-p^k\zu$. Therefore, we may apply Corollary \ref{BSlemma} c) and d) and $\beta=p^k$, which gives
         \begin{align}
            1-\reih{n}{X}+\reih{2n}{X}-\reih{3n}{X}\equiv0\mod p\,,\label{eq:d2.1}
        \end{align}
       as well as
        \begin{align}
            \reih{n-p^k}{X}-\reih{2n-p^k}{X}+\reih{3n-p^k}{X}\equiv0\mod p\,.\label{eq:e2.1}
        \end{align}
Subtracting \eqref{eq:d2.1} and adding \eqref{eq:e2.1} to \eqref{eq:321clean}, then using the fact that $\reih{n}{X}=\reih{3n}{X}=0$, it follows from \eqref{eq:n-cvs3n-c} that
        \begin{align*}
        0\equiv& \ 2-\reih{n-p^k}{X}+\reih{3n-p^k}{X}\\
        \equiv& \ 2\mod p \,
    \end{align*}
    so that $p=2$, which is a contradiction. 
    
    \medskip

    Now, suppose that $\D(G)\leq n+p^k$. In this case, we have $|X|\in\auf\D(G)+2n-1-p^k,\ 3n-1\zu$, and since $n+1 \le \D(G)$ as well as $\reih{n}{X}=0$, we can apply Lemma \ref{binom_pte_dav} which gives 
      \begin{align}
        2-2\reih{n-p^k}{X}+\reih{2n-p^k}{X}\equiv0\mod p\,.\label{eq:221}
    \end{align}
    Finally, as in the previous case, we may apply Corollary \ref{BSlemma} d) for $\beta=p^k$ which, taking \eqref{eq:n-cvs3n-c} into account, gives us 
    \begin{align}
        2\reih{n-p^k}{X}\equiv\reih{2n-p^k}{X}\mod p.
    \end{align}
    Injecting this in \eqref{eq:221} gives 
      \begin{align*}
        2\equiv0\mod p\,,
    \end{align*}
    so that $p=2$, which is a contradiction.
    \pushQED{}\qed\popQED 

\medskip

Let us now prove the announced corollaries of Theorem \ref{thm:thmgal}.

\begin{proof}[Proof of Corollary \ref{cor:sbestimmen}]
    We set $n=\exp(G)$. Injecting $\D(G)=2n-p^k$ in the upper bound given by Theorem \ref{thm:thmgal}, we obtain
    \begin{align*}
        \s(G)\leq\D(G)+2n-p^k-1=2\D(G)-1.
    \end{align*}
The reverse inequality is just (\ref{boundsp}).
\end{proof}

\begin{proof}[Proof of Corollary \ref{cor:<}]
    We set $n=\exp(G)$. In view of Theorem \ref{thm:thmgal}, it suffices to show that there is no finite abelian $p$-group $G$ of exponent $n$ satisfying $2n-p <\D(G) < 2n-1$.
    To show this, note first that by Theorem \ref{Olson1}, we have
    \begin{align*}
        \D(G)=\sum_{i=1}^r(p^{a_i}-1) +1=\sum_{i=1}^r(p-1)s_i + 1,
    \end{align*}
    where $s_i=\displaystyle\sum_{j=0}^{a_i-1}p^j$ for every $i \in \interv{1}{r}$. Therefore,
    \begin{align*}
        \D(G)\equiv1\mod (p-1)\,.
    \end{align*}
    However, since $2n-p$ and $2n-1$ are consecutive elements in $1+(p-1)\mathbb{Z}$, it follows that $\interv{2n-p+1}{2n-2}$ contains no integer congruent to $1 \mod (p-1)$. 
\end{proof}

Let us now proceed with the proof of Theorem \ref{thm:pgpe+cycl} and Corollary \ref{cor:pgpe+cycl}. This is essentially the same proof as the one of Theorem $4.1$ in \cite{luo}, which uses the following two classical lemmas (see Corollary $4.2.13$ and Lemma $4.2.5$ in \cite{rusza}).

\begin{lemma}\label{lm:pgpeextended}
    Let $H$ be a finite abelian $p$-group such that $\D(H) \le 2\exp(H)-1$, and $a$ be a positive integer coprime to $p$. Then, $G \simeq H \oplus C_a$ satisfies
    \begin{align*}
        \D(G)=\D(H)+\exp(H)(a-1)\,.
    \end{align*}
\end{lemma}

\begin{lemma}\label{cor:quot}
    Let $G$ be a finite abelian group, and $L$ be a subgroup of $G$ such that $\exp(G)=\exp(L)\exp(G/ L)$. Then, 
    \begin{align}
        \s(G)\leq(\s(L)-1)\exp(G/L)+\s(G/L)\,.
    \end{align}
\end{lemma}

\medskip
\noindent
\emph{Proof of Theorem \ref{thm:pgpe+cycl} and Corollary \ref{cor:pgpe+cycl}.}
    Let us set $n=\exp(H)$ and $L=C_a$. We have $G/L \simeq H$, so that $\exp(G/L)=\exp(H)=n$, as well as $\exp(L)=a$ and $\exp(G)=an$. Thus, Lemma \ref{cor:quot} applies and gives
    \begin{align*}
        \s(G) \leq n(\s(L)-1)+\s(G/L).
    \end{align*}
    Now, it follows from the Erd\H{o}s-Ginzburg-Ziv theorem (Theorem \ref{ranktwo} with $m=1$ and $n=a$) that $\s(C_a)=2a-1$, and from Theorem \ref{thm:thmgal} that $\s(G/L) = s(H) \le \D(H)+2n-p^k-1$. Therefore,
    \begin{align*}
        \s(G)& \leq n(2a-2)+\D(H)+2n-p^k-1\,\\
        & =\D(H)+2an-p^k-1,
    \end{align*}
    which is the upper bound claimed in Theorem \ref{thm:pgpe+cycl}.
    
    Moreover, if $\D(H)=2n-p^k$, and since $\D(G)=\D(H)+n(a-1)$ by Lemma \ref{lm:pgpeextended}, it is easily checked that the upper bound we just obtained coincides with $2\D(G)-1$. Finally, writing $H=K\oplus\bZ_n$, we have $G \simeq K \oplus C_{na}$ and it follows from Lemma $3.2$ in \cite{edel} that $\s(G)\geq 2(\D(K)-1)+2an-1$. By Theorem \ref{Olson1}, we have $\D(K)=\D(H)-n+1$, whence $\s(G)\geq 2(\D(H)-n)+2an-1=2\D(G)-1$. This completes the proof of Corollary \ref{cor:pgpe+cycl}.
\pushQED{}\qed\popQED

\section{Proof of Theorem \ref{thm:thmc}}
\label{extraresult}

In this section, we push the method we used to prove Theorem \ref{thm:thmgal} a little further. From a technical point of view, the argument is more involved, but once again, Corollary \ref{BSlemma} and Lemma \ref{ABC} will complement each other in order to give the desired result.

\begin{proof}[Proof of Theorem \ref{thm:thmc}]
    Set $n=\exp(G)$. First of all, since $c=n$ gives $\D(G)=n$ in which case $G$ is cyclic and satisfies the claimed upper bound by the Erd\H{o}s-Ginzburg-Ziv theorem, we may assume that $c<n$. 
    In addition, since $p$ is odd and $\D(G) \equiv 1 \mod (p-1)$ also, it follows that $n$ and $c$ are odd too. Now, let us set $c'=(\frac{c-1}{2})+1$, consider a sequence $X \in \cF(G)$ of length 
    $|X| = \D(G)+2n-\left(\frac{c-1}{2}\right)-2=\D(G)+2n-c'-1$, and assume for a contradiction that $\reih{n}{X}=0$. 
    
    Let $J \mid X$ be any subsequence of length $|J|=|X|-n=\D(G)+n-c'-1$. On the one hand, since $\D(G)+n-c'-1\leq3n-1-c'$, Corollary \ref{BSlemma} a) with $\gamma=c'$ applies to $J$ and yields
    \begin{equation}\label{5dot0}\tag{$E_0$}
0 \equiv
\,  1-\left(\sum_{i=0}^{c'}\binom{c'}{i}\reih{n-i}{J}\right)+\left(\sum_{i=0}^{c'}\binom{c'}{i}\reih{2n-i}{J}\right)\mod p\,, 
\end{equation}
On the other hand, since $\D(G)+n-1-c'=3n-1-c'-c\leq3n-1-c'-(c'-1)$, Corollary \ref{BSlemma} b) with $\gamma=c'$ and any $\beta\in\interv{1}{c'-1}$ applies to $J$ and gives,
\begin{equation}\label{5dotj}\tag{$E_j$}
0 \equiv -\left(\sum_{i=0}^{c'}\binom{c'}{i}\reih{n-i-j}{J}\right)+\left(\sum_{i=0}^{c'}\binom{c'}{i}\reih{2n-i-j}{J}\right) \mod p\,, 
\end{equation}
 for every $j\in\interv{1}{c'-1}$.

\medskip
These $c'$ equations (\ref{5dot0}) and (\ref{5dotj}), for $j=1,\dots,c'-1$, correspond to a linear system $AX=0$ over $\mathbb{F}_p$, which is satisfied by the vector $X_J=(\reih{0}{J},\dots,\reih{2n}{J})^{\mathsf{T}}$. In particular, $A$ can be seen as a matrix indexed by $\interv{1}{c'} \times \interv{0}{2n}$. Now, for every $L \subseteq \interv{0}{2n}$, let us write $A_L$ for the submatrix of $A$ obtained from $A$ by deleting all columns $C_{\ell}$ such that $\ell \notin L$.

Now note that, since $c < n$, one has $n < 2n-2c'+1$. Let us also set
$$\begin{tikzpicture}
    \matrix[
      matrix of math nodes,
      left delimiter=(,
      right delimiter=),
      nodes in empty cells
    ] (m){
      {c' \choose c'-1}     & {c' \choose c'-2} & \cdots  & \cdots & {c' \choose 1}    \\
{c' \choose c'-2}     &    &  & {c' \choose 1}  & 1     \\
\vdots     &     & &       & 0     \\
\vdots   &    &     &   &   \\
{c' \choose 1}  &   1  &   &  &    \\
1  & 0  & &     & 0     \\
    };
    \draw (m-5-1) -- (m-2-4);
        \draw (m-6-2) -- (m-3-5);
    \draw (-0.35,-1.05) -- (1.85,0.725);
     \draw[fill] (0.79,1.91) circle (0.415pt);
          \draw[fill] (-1.905,-0.1825) circle (0.415pt);
               \draw[fill] (-1.905,-0.357) circle (0.415pt);
               \draw[fill,white] (-1.905,-0.82) circle (1pt);
              \draw (-3,0) node [anchor=east]  {$B=$};
    \draw (m-3-5) -- (m-6-5);
    \draw (m-6-2) -- (m-6-5);
  \end{tikzpicture}
$$
It is readily seen that $A_{\interv{n-c'+1}{n-1}}=-B$ and $A_{\interv{2n-c'+1}{2n-1}}=B$. It is all the more easy to check that the matrix $B^\mathsf{T}$ has size $(c'-1) \times c'$ and rank $c'-1$. Indeed, the $(c'-1) \times (c'-1)$ submatrix obtained from $B^\mathsf{T}$ by deleting its first column is invertible. Therefore, it follows from the rank-nullity theorem that there exists $(\lambda_1,\dots,\lambda_{c'}) \in \mathbb{F}^{c'}_p$ such that $\lambda_1 \neq 0$ and $(\lambda_1,\dots,\lambda_{c'})B^\mathsf{T}=0$.

\medskip
Multiplying both sides of the equality $AX_J=0$ to the left by $(\lambda_1,\dots,\lambda_{c'})$ yields a new identity of the form
    \begin{align}
        0 \ \equiv \ &1-a_{2c'-1}\reih{n-2c'+1}{J}-\dots-a_{c'}\reih{n-c'}{J}-\reih{n}{J}\notag\\
        &+a_{2c'-1}\reih{2n-2c'+1}{J}+\dots+a_{c'}\reih{2n-c'}{J}+\reih{2n}{J}\notag\\
        &\hspace{200pt}\mod p\,.\label{eq:eqnextrathm2}
    \end{align}
    for some coefficients $a_{c'},\dots,a_{2c'-1}$ in $\mathbb{F}_p$.

\medskip
    Summing up \eqref{eq:eqnextrathm2} over all subsequences $J\mid X$ of length $|J|=|X|-n$ and finally taking into account that $\reih{n}{X}=0$ gives
    \begin{align}
        0\equiv \ &\binom{|X|}{n}-\left(\sum_{i=c'}^{2c'-1}a_i\binom{|X|-n+i}{n}\reih{n-i}{X}+\binom{|X|-n}{n}\reih{n}{X}\right)\notag\\
        & +\left(\sum_{i=c'}^{2c'-1}a_i\binom{|X|-2n+i}{n}\reih{2n-i}{X}+\binom{|X|-2n}{n}\reih{2n}{X}\right)\notag\\
        =& \ \binom{|X|}{n}-\left(\sum_{i=c'}^{2c'-1}a_i\binom{|X|-n+i}{n}\reih{n-i}{X}\right)\notag\\
        &+\left(\sum_{i=c'}^{2c'-1}a_i\binom{|X|-2n+i}{n}\reih{2n-i}{X}+\binom{|X|-2n}{n}\reih{2n}{X}\right)\label{eq:binomextrathm}\,.
    \end{align}
    Now, consider two cases: If $c+c'+1\le n$, we have $|X|=4n-c-c'-1\in\interv{3n}{4n-1}$ (whence we have of course $|X|-n\in\interv{2n}{3n-1}$ and 
    $|X|-2n\in\interv{n}{2n-1}$). So by Remark \ref{rq:binommodp}, we have $\binom{|X|}{n}\equiv3\mod p$
    and $\binom{|X|-2n}{n}\equiv1\mod p$. 

    \medskip
    If $c+c'\ge n$, then we have $|X|=4n-c-c'-1\in\interv{2n}{3n-1}$ (as clearly $c+c'\leq2n-1$) and so by Remark \ref{rq:binommodp} we have $\binom{|X|}{n}\equiv2\mod p$ and $\binom{|X|-2n}{n}\equiv0\mod p$.

    \medskip
     On the other hand, and since $c\leq n-1$, we have in both cases that $2n\leq3n-c-1=|X|-n+c'\leq|X|-n+i\leq|X|-n+2c'-1=3n-c+c'-2\le3n-1$ for all $i\in\interv{c'}{2c'-1}$, so again by Remark \ref{rq:binommodp}, $\binom{|X|-n+i}{n}\equiv2\mod p$ and  $\binom{|X|-2n+i}{n}\equiv1\mod p$. Therefore, \eqref{eq:binomextrathm} gives the equation 
    \begin{align}
        0\equiv \ 3-\left(\sum_{i=c'}^{2c'+1}2a_i\reih{n-i}{X}\right)+\left(\reih{2n}{X}+\sum_{i=c'}^{2c'+1}a_i\reih{2n-i}{X}\right)\label{eq:binomextrathmcas1}
    \end{align}
    in the first case and the equation
    \begin{align}
        0\equiv \ 2-\left(\sum_{i=c'}^{2c'+1}2a_i\reih{n-i}{X}\right)+\left(\sum_{i=c'}^{2c'+1}a_i\reih{2n-i}{X}\right)\label{eq:binomextrathmcas2}
    \end{align}
    in the second case.

    \medskip

    We now want to apply Lemma \ref{ABC} to every $\alpha=i\in\interv{c'}{2c'-1}$. Let us fix such an $i$. By definition of $c'$, we have $i\leq2c'-1\leq c+1=2n+1-\D(G)$, and so it suffices to check that $|X|=\D(G)+2n-c'-1\in\auf\D(G)+2n-i-1,4n-i-1\zu$ to have the hypotheses of the lemma verified. We do have indeed $\D(G)+2n-c'-1\in\auf\D(G)+2n-c'-1,4n-(2c'-1)-1\zu$ (and this interval is contained in the desired one, by definition of $i$), as by definition of $c'$ we have $\D(G)=2n-c\leq2n-c'+1$, so we can actually apply the lemma which gives us
    \begin{align}
        \reih{n-i}{X}\equiv\reih{3n-i}{X}\mod p\,.\label{eq:n-i,3n-i}
    \end{align} 
    Moreover, as we have $|X|=\D(G)+2n-c'-1\in\auf\D(G)+n-1,4n-1-i\zu$ (this is true because of $c'\leq n$ and $\D(G)+2n-c'-1=4n-c-c'-1\leq4n-1-(2c'-1)$, which is again true in view of $c'-1\leq c$), we may also apply Corollary \ref{BSlemma} d) for $\beta=i$, whence we have, together with \eqref{eq:n-i,3n-i}, 
    \begin{align}
        2\reih{n-i}{X}\equiv\reih{2n-i}{X}\mod p\,.\label{eq:n-i,2n-i}
    \end{align}

    In the first case, injecting \eqref{eq:n-i,2n-i} for each $i\in\interv{c'}{2c'-1}$ in \eqref{eq:binomextrathmcas1} gives 
    \begin{align*}
        \reih{2n}{X}\equiv-3\mod p\,.
    \end{align*}
    Yet, on the other hand, one has by Corollary \ref{BSlemma} c) (which may be applied, as we already saw that $|X|\in\auf\D(G)+n-1,4n-1\zu$) and Lemma \ref{Dubiner} 
    \begin{align*}
        \reih{2n}{X}\equiv-1\mod p\,,
    \end{align*}
    whence we deduce
    \begin{align*}
        0\equiv2\mod p\,,
    \end{align*}
    so that $p=2$, which is a contradiction.
    
    \medskip
    
    In the second case, injecting \eqref{eq:n-i,2n-i} for each $i\in\interv{c'}{2c'-1}$ in \eqref{eq:binomextrathmcas2} yields
    \begin{align*}
        0\equiv2\mod p,
    \end{align*}
    so that $p=2$, which is a contradiction again.
\end{proof}

\section{Concluding remarks}
\label{conclu}

Over the years, many variants of the Erd\H{o}s-Ginzburg-Ziv constant have been introduced and studied. For instance, given any finite abelian group $G$ of exponent $n$, one can consider the invariant $\s_{\interv{j}{n}}(G)$, for every $j\in\interv{1}{n}$. 

\medskip
Note that this quantity acts as a common generalization of $\eta(G)=\s_{\interv{1}{n}}(G)$ and $\s(G)=\s_{\interv{n}{n}}(G)$. In addition, Luo observed (see Section $5$ in \cite{luo}) that the sequence $(\s_{\interv{j}{n}}(G))_{j \in \interv{1}{n}}$ 
is strictly increasing. This fact has the following consequence.

\begin{corollary}\label{sj}
    Let $G$ be a finite abelian group of exponent $n$. Then, for every $j\in\interv{1}{n}$, one has $$s_{\interv{j}{n}}(G)\in\interv{\eta(G)+j-1}{\s(G)-n+j}.$$ 
    In particular, if $G$ satisfies Conjecture \ref{ConjGao}, then for every $j\in\interv{1}{n}$, one has
    $$s_{\interv{j}{n}}(G)=\eta(G)+j-1\,.$$
\end{corollary}
\begin{proof}
    Since $(\s_{\interv{j}{n}}(G))_{j \in \interv{1}{n}}$ is strictly increasing, one has the inequalities
    \begin{align*}
        \s(G)&=\s_{\interv{n}{n}}(G)\\
        &\geq\s_{\interv{n-1}{n}}(G)+1\geq\s_{\interv{n-2}{n}}(G)+2\geq\cdots\geq\s_{\interv{1}{n}}(G)+n-1\\
        &=\eta(G)+n-1\,,
    \end{align*}
    from which the claimed result directly follows.
\end{proof}
Combining Corollary \ref{sj} with Theorem \ref{thm:pgpe+cycl} yields the following general result.

\begin{theorem}\label{thm:thmgalgal}
Let $p \ge 3$ be a prime. Let also $H$ be a finite abelian $p$-group of exponent $n$ such that $\D(H) \le 2n-p^k$ for some prime power $1 \le p^k \le n$, and $a$ be a positive integer coprime to $p$. Then, the group $G \simeq H\oplus\bZ_{a}$ satisfies $\exp(G)=an$ and, for every $j\in\interv{1}{an}$, one has 
    \begin{align*}
        2\D(G)-an+j-1\leq\s_{\interv{j}{an}}(G)\leq\D(H)+an-p^k+j-1\,.
    \end{align*} 
\end{theorem}

In particular, Theorem \ref{thm:thmgalgal} and Corollary \ref{cor:pgpe+cycl} give the exact value of all quantities $s_{\interv{j}{n}}(G)$ whenever $G$ is a direct product of a finite abelian $p$-group $H$ such that $\D(H)=2n-p^k$ for some prime power $1 \le p^k \le n$ with a cyclic group of order coprime to $p$. 

\begin{corollary}\label{corogal}
        Let $p \ge 3$ be a prime. Let also $H$ be a finite abelian $p$-group of exponent $n$ such that $\D(H) = 2n-p^k$ for some prime power $1 \le p^k \le n$, and $a$ be a positive integer coprime to $p$. Then, the group $G \simeq H\oplus\bZ_{a}$ satisfies $\exp(G)=an$ and, for every $j\in\interv{1}{an}$, one has
    $$s_{\interv{j}{an}}(G)=2\D(G)-an+j-1\,.$$
\end{corollary}

In the special case of $p$-groups, Corollary \ref{corogal} is consistent with the following conjecture of Luo, bearing upon the structure of long sequences over rank-two-like $p$-groups (see Conjecture 5.4 in \cite{luo}). 

\begin{conjecture}
    Let $p$ be a prime, and $G$ be a finite abelian $p$-group of exponent $n$ such that $\D(G)\leq2n-1$. Let also $\ell\in\interv{1}{\D(G)+1-n}$. If $S\in\cF(G)$ is a sequence of length $|S| \ge \D(G)+n-2+\ell$, then one of the following two statements holds.
    \begin{enumerate}
        \item[$(i)$] $S$ contains a zero-sum subsequence of length $n$.
        \item[$(ii)$] $S$ contains a zero-sum subsequence $T \mid S$ of length $2n$ containing itself a zero-sum subsequence $U \mid T$ of length $|U| \in\interv{2n-1-\D(G)+\ell}{n-1}$.
    \end{enumerate}
\end{conjecture}
This conjecture is currently known to hold when $\ell=1$ (see Theorem $5.2$ in \cite{luo}), and for every $\ell\in\interv{1}{\D(G)+1-n}$ under the stronger assumption that $\D(G)=2n-1$ (see Theorem $3.1.2$ in \cite{zhuang}).

\end{document}